\documentclass[10pt]{amsart}
\usepackage{amsthm, amsmath, hyperref, cleveref, amsfonts, amssymb, enumerate, textcmds, tensor, enumitem, mathdots}
\usepackage{pst-node}
\usepackage{tikz-cd} 
\usepackage{graphicx,tikz}
\usepackage{enumerate}
\usepackage{pinlabel}
\usepackage[colorinlistoftodos]{todonotes}
\usepackage{commath}
\usepackage[Bourbaki-arrow]{dynkin-diagrams}
\usepackage{placeins}
\usepackage{mathtools}

\let\svthefootnote\thefootnote
\newcommand\freefootnote[1]{%
  \let\thefootnote\relax%
  \footnotetext{#1}%
  \let\thefootnote\svthefootnote%
}

\newcommand{\R}{\mathbb R}

\newtheorem{thm}{Theorem}[section]
\newtheorem{lem}[thm]{Lemma}

\newtheorem{cor}[thm]{Corollary}

\DeclareMathOperator{\C}{\mathbb{C}}

\allowdisplaybreaks

\theoremstyle{definition}

\newtheorem{lemma}[thm]{Lemma}

\newtheorem{remark}[thm]{Remark}

\begin{document}
\title{Holomorphicity of stable minimal surfaces of low genus}
\author{Nathaniel Sagman}
\address{Nathaniel Sagman, University of North Carolina at Chapel Hill, NC, United States}
\email{nsagman@unc.edu}
\author{Thomas-Ren\'e Thalmaier}
\address{Thomas-Ren\'e Thalmaier, Institut Polytechnique de Paris, Palaiseau, France}
\email{thr.thalm@gmail.com}

\begin{abstract}
We prove that a (branched) minimal immersion from $\C$ to $\R^n$ is stable if and only if it lives in an even dimensional affine subspace and is holomorphic for some orthogonal complex structure on the subspace. More generally, we prove that the same result holds for a class of genus $0$ surfaces that can have infinite total curvature. This contributes to an inquiry initiated by Micallef, who previously proved the equivalence in genus $0$ assuming completeness and finite total curvature. As a corollary, we prove a holomorphicity result for covering stable minimal surfaces of genus $0$ and $1$, recovering a theorem of Fraser and Schoen as a particular case.

Our approach is new, based on a method of constructing variations developed by the first named author and Marković. For unstable surfaces, we get explicit destabilizations and destabilization radii that can be read from the Weierstrass-Enneper data.
\end{abstract}
\maketitle

\section{Introduction}
A minimal surface in $\R^n$ is said to be area minimizing if it minimizes area when restricted to compact subsets. It is well understood that since holomorphic curves in $\C^{n}$ are calibrated, they are area minimizing as surfaces in $\R^{2n}$. In light of this, Micallef wrote in \cite{Micallef}, ``it is reasonable to ask whether an area minimizing surface in $\R^n$ lies in an even dimensional affine subspace of $\R^n$ and is holomorphic with respect to some orthogonal complex structure on this even dimensional affine subspace."  Micallef's question reflects the perspective that, inside a Riemannian manifold, minimal surfaces or other geometrically meaningful immersed submanifolds should respect and reveal the geometry of the ambient space.

The holomorphicity condition is equivalent to saying the minimal surface is parametrized by a map of the form $T\circ F$, where $F$ is a holomorphic map to $\C^{m}\subset \R^{n}$, $m= \lfloor \frac{n}{2} \rfloor$, and $T$ is a rigid motion of $\R^{n}$ (see also Theorem \ref{thm:main2}). For brevity, we'll say that such a minimal surface is \textit{holomorphic up to a rigid motion}, or just \textit{holomorphic} when the context is clear. For $n=3$, the condition always holds due to Fischer-Colbrie and Schoen \cite{FCS} and do Carmo and Peng \cite{dCP}, who showed that any stable complete minimal surface in $\R^3$ is an affine plane. In \cite{Micallef}, among other things, Micallef proved the following two positive results. Recall that ``stable" means area minimizing on compact sets up to second order.
\begin{itemize}
    \item When $n=4$, any stable minimal immersion of a
complete oriented parabolic surface into $\R^4$ is holomorphic \cite[Theorem I]{Micallef} (extended to branched immersions in \cite{Micallefnote}).
\item For all $n$, any stable minimal immersion of a complete oriented  surface of finite total curvature and genus $0$ into $\R^n$ is holomorphic \cite[Theorem IV]{Micallef}. By Chern and Osserman \cite[Theorem 1]{CO}, any such surface is a finitely puncture sphere.
\end{itemize}

Since \cite{Micallef}, there have been a number of developments. Notably, in \cite{Ar}, Arezzo, Micallef, and Pirola produced once-punctured genus $2$ minimal surfaces in $\R^{22}$ that are not holomorphic (see also \cite{AM_tori} on minimal surfaces in flat tori). More recently, Fraser and Schoen proved that any complete oriented genus 1 minimal surface of finite total curvature in $\R^n$, if covering stable (i.e., stable under finite covers), is holomorphic \cite[Theorem 1.1]{FS}, and Cheng, Karigiannis, and Madnick proved that a complete oriented parabolic stable minimal surface in $\R^{2+2k}$ is holomorphic if and only if the induced connection on the normal bundle has holonomy in $U(k)$ \cite{CKM}. 

In this paper, we make another contribution to Micallef's inquiry. With a method of proof that comes from a new perspective, we prove that the answer to his question is yes for certain genus $0$ surfaces that can have infinite total curvature (Theorems \ref{thm: main_C_case} and \ref{thm: main}), recovering \cite[Theorem IV]{Micallef} as a special case. We use these results to prove that, for certain genus $0$ and genus $1$ surfaces, the answer is yes when we replace ``stable" with ``covering stable" (Corollary \ref{thm: FS}); as a particular case, we get a new proof of \cite[Theorem 1.1]{FS}. Our proofs lead to estimates on destabilization radii of non-holomorphic maps, which we make explicit in the case of maps from $\C$ to $\R^3$ (Theorem \ref{thm: radius}).

Our first main theorem concerns minimal surfaces with the simplest topology, with no assumption on the geometry.
\begin{thm}\label{thm: main_C_case}
    A branched minimal immersion $F:\C\to \R^n$ is stable if and only if it is holomorphic up to a rigid motion.
\end{thm}
For context, by \cite{CO}, an isometric minimal immersion of $\C$ is complete with finite total curvature (in which case, the result is contained in \cite[Theorem IV]{Micallef}) if and only if every component of the map is a polynomial (see also Remark \ref{rem: drop completeness}). Although Theorem \ref{thm: main_C_case} deals with planes, it sheds light on the situation for minimal surfaces with arbitrary topology: it suggests that if a minimal surface contains a disk of large conformal radius, then we should suspect that stability is equivalent to holomorphicity.

Our second main theorem treats certain genus $0$ surfaces with infinitely many ends. This theorem in fact contains Theorem \ref{thm: main_C_case} (below, take $P=\emptyset$), but we separated them because Theorem \ref{thm: main_C_case} is cleaner and worth emphasizing. We say that a minimal surface is locally complete at an end if any curve going into the end has infinite length in the induced metric. We say that a branched minimal immersion of a punctured surface is locally complete with finite total curvature around the punctures if the corresponding ends have those properties. The set of ends at which a minimal surface is locally complete with finite total curvature is discrete (by a minor extension of \cite{CO}, see \S \ref{sec: reducing}), and hence any genus $0$ surface with infinite ends has an end of infinite total curvature. 
\begin{thm}\label{thm: main}
Let $\Sigma=\C\backslash P,$ where $P$ is a discrete subset, and let $F:\Sigma\to \R^n$ be a branched minimal immersion that is locally complete with finite total curvature around $P$. Then $F$ is stable if and only if it is holomorphic up to a rigid motion.
\end{thm}
Note that there is no completeness or curvature assumption for the end at $\infty$. We can also prove the theorem for certain minimal immersions with infinite curvature at multiple ends, but we need to impose a technical condition. The result is Theorem \ref{thm: main_extended}, proved in \S \ref{sec: reducing}. It would be interesting to see if Theorem \ref{thm: main} can be pushed further; the results from \cite{MS} that we use (see below) don't seem to be sufficient to do so, not at least without further development. 

Theorem \ref{thm: main} extends \cite[Theorem IV]{Micallef}. We relax the completeness (by permitting branching, and not requiring completeness at one end) and curvature conditions, and allow for infinite ends. Regarding completeness, we point out that, using the second variation formula from \cite{Micallefnote}, the original proof of \cite[Theorem IV]{Micallef} pushes through for branched minimal immersions with finite total curvature at punctures.
By contrast, the finite total curvature hypothesis is essential to the proof in \cite{Micallef}: the starting point is that the generalized Gauss map of a finite total curvature minimal surface extends holomorphically over the ends of the surface. 

About removing curvature conditions, on one hand, an extension of \cite[Theorem IV]{Micallef} to infinite total curvature minimal surfaces might not be surprising, since such a surface is wild at an end and is unlikely to be stable to begin with. On the other hand, holomorphic maps can have infinite total curvature, and at an essential singularity there might be enough freedom to perturb to a non-holomorphic but stable minimal surface (as is done in \cite{Ar}). Theorem \ref{thm: main_C_case} shows that this is not the case.

As our final result on holomorphicity, we use Theorems \ref{thm: main_C_case} and \ref{thm: main} to study minimal surfaces that, after filling in punctures, are universally covered by $\C$. Recall that a minimal immersion $F:X\to \R^n$ is covering stable if for any finite covering $p: Y\to X$, the minimal immersion $F\circ p: Y\to \R^n$ is stable. If $F$ is holomorphic for the induced complex structure on $X$, then $F\circ p$ is holomorphic with respect to the pullback complex structure on $Y$, and hence holomorphic maps are covering stable.
\begin{cor}\label{thm: FS}
Let $\Sigma$ be either $\C^*$ or a compact oriented surface of genus $1$. For $P\subset \Sigma$ a discrete subset, let $F:\Sigma\backslash P\to \R^n$ be a branched minimal immersion that is locally complete with finite total curvature around $P$. Then $F$ is covering stable if and only if it is holomorphic up to a rigid motion. 
\end{cor}
Taking $\Sigma$ to be compact of genus $1$ and $F$ to be an isometric immersion, we obtain Theorem 1.1 from \cite{FS}. Another special case of Corollary \ref{thm: FS} comes from taking $\Sigma=\C^*$ and $P=\emptyset$. Then we get that any covering stable branched minimal immersion from $\C^*$ to $\R^n$ is holomorphic up to a rigid motion. This result is not contained in Theorem \ref{thm: main}, since we could have infinite curvature at $0$. We should point out that we are not aware of any minimal surface as in Corollary \ref{thm: FS} that is stable but not covering stable.

\begin{remark}\label{rem: drop completeness}
To further understand the scope of Theorem \ref{thm: main_C_case}, if we assume finite total curvature but don't impose completeness, then the Weierstrass-Enneper data is of the form $(f\alpha_1,\dots, f\alpha_n)$, where $\alpha_1,\dots, \alpha_n$ are polynomial holomorphic $1$-forms and $f$ is an arbitrary holomorphic function (see the proof of Theorem 4.1.1 from \cite{AlarconLopezForstneric2021}).
\end{remark}

\begin{remark}\label{rem: stability}
Note we're considering stability of a map, rather than stability of a surface (see \S \ref{sec: basics} for definitions). In particular, ``$F$ is stable" means that $F$ is stable with respect to variations of the map, including those that might change the topology of $F(\Sigma)$. The example to have in mind is a branched minimal immersion $F: \C\to \R^n$ that factors through the standard covering $\C\to \C^*$; variations might change the topology of $F(\C)$ from a cylinder to a plane, and hence we have not proved that the minimal cylinder is stable if and only if it's holomorphic. Of course, for isometric minimal immersions, the notions of stability coincide.
\end{remark}
\begin{remark}
    Lawson gave a characterization of the generalized Gauss map of a holomorphic minimal surface in \cite[Proposition 1.6]{Law}.
\end{remark}

\subsection{On the proofs}
One of the key insights from \cite{Micallef} is that if one can find splittings of the complexified tangent and normal bundles satisfying certain properties, then the minimal immersion is holomorphic (see \cite[Theorem A]{Micallef}). The proofs of \cite[Theorem IV]{Micallef} and \cite[Theorem 1.1]{FS} proceed by constructing such splittings, using the extended Gauss map to extend the normal bundle over the compactified surface and then relying on known classifications of holomorphic vector bundles over closed surfaces of genus $0$ and $1$. Our proof of Theorem \ref{thm: main_C_case} is totally different. We make use of the Weierstrass-Enneper data of minimal surfaces and a destablization strategy for minimal surfaces introduced by the first named author and Markovi{\'c} in \cite{MS} (see \S \ref{sec: destabilizing} for an overview of the method). Briefly, in \cite{MS}, the space of variations of a minimal disk is identified with the tangent space at the identity of the universal Teichm{\"u}ller space. This leads to a model for computation that gives a new way to construct and study interesting variations.

To prove Theorem \ref{thm: main_C_case}, we first observe that holomorphicity is equivalent to an isotropy condition on the holomorphic jets of $F$ (Theorem \ref{thm:main2}). We then use \cite[Theorem C]{MS} to produce, for each $N$, explicit compactly supported variations of $F(\C)$ that have certain homogeneity properties and rotational symmetries. Due to basic Fourier orthogonality, these variations see the holomorphic jets only up to order $N$ (see Theorem \ref{thm: thmc}, Lemma \ref{lem: main computation}). We use these variations to prove by induction on $N$ that a stable map satisfies the isotropy condition. Using \cite{MS}, the computation is elementary and tractable (but delicate). Notably, larger growth of the minimal surface and curvature at infinity do not lead to conceptual differences in the proof.

Theorem \ref{thm: main} is proved using Theorem \ref{thm: main_C_case}. The completeness and curvature conditions imply that the Weierstrass-Enneper data $(\alpha_1,\dots, \alpha_n)$ for $F:\C\backslash P \to \R^n$ is meromorphic at $P$. By Weierstrass factorization, there exists a holomorphic function $f$ such that $(f\alpha_1,\dots, f\alpha_n)$ extends holomorphically to $\C$ and hence determines a new branched minimal immersion $F_f:\C\to \R^n$, which has the same generalized Gauss map as $F$. The holomorphicity of $F$ and $F_f$ are easily seen to be equivalent, and by the log cut-off trick, the stability is equivalent too. Essentially, replacing $F$ with $F_f$ pushes all of the singular behaviour to the end at $\infty$, while preserving properties of interest, and allows us to apply Theorem \ref{thm: main_C_case}.

To prove Corollary \ref{thm: FS}, we use that $F(\Sigma\backslash P)$ is covered by a minimal surface coming from Theorem \ref{thm: main}. Hence, if $F$ is not holomorphic, the minimal surface has an unstable cover. This does not immediately prove the corollary, because the covering map has infinite degree. However, the covering can in some sense be approximated by finite covers, and with this we're able to fit an unstable subsurface into a sufficiently large finite cover.

Going back to Theorem \ref{thm: main_C_case}, since the variations resulting from \cite{MS} play so well with the isotropy condition, we hope that the ideas introduced here will be of further use in studying Micallef's question, and for related questions about more general minimal surfaces.

\subsection{Explicit destabilization radii}
A branched minimal immersion $h:\C\to \R^n$ is unstable if and only if there is a minimal $r_0>0$ such that for $r>r_0$, $h|_{\overline{\mathbb{D}}_r}$ is unstable as a branched minimal immersion of a surface with boundary. Here, $\mathbb{D}_r=\{z\in \C: |z|<r\}$.  We refer to this $r_0$ as the \textit{destabilization radius} of $h$. The method of \cite{MS} and the proof of Theorem \ref{thm: main_C_case} give a way to compute or estimate the destabilization radius by studying zeros of polynomials. The method could also be applied on a minimal surface of any topological type that contains a large isometrically embedded disk.

We thought it worthwhile to demonstrate an estimate of the destabilization radius in the simplest case: for maps from $\C$ to $\R^3$. Carrying out a simplified version of the proof of Theorem \ref{thm: main_C_case}, we arrive at the theorem below. For the statement, recall that via the Weierstrass-Enneper representation, a branched minimal immersion on $\C$ is equivalent to a holomorphic function $f$ and a meromorphic function $g$ such that $fg^2$ is holomorphic. The function $g$ is the stereographic projection of the Gauss map. 

\begin{thm}\label{thm: radius}
    Let $r>0$ and let $h:\overline{\mathbb{D}}_r\to \R^3$ be a branched minimal immersion determined by Weierstrass-Ennerper representation $(f,g)$, where 
    $$f(z) = \sum_{j=k}^\infty b_j z^j, \hspace{1mm} g(z)=a_0 + \sum_{j=m}^\infty a_j z^j,$$ with $m>0$, $a_m, b_k\neq 0$. If $r$ is strictly larger than the smallest positive root of the polynomial
\begin{equation}\label{eq:polynomial_holomorphic_case}
    P_m(r):=-m\vert b_ka_m\vert^2r^{2m}+\sum_{j=0}^{m-1}(m-j)\vert b_{k+j}\vert^2(1+\vert a_0\vert^2)^2r^{2j},
\end{equation}
then $h$ is unstable.
\end{thm}
Note that $P_m(0)=0$ and $\lim_{r\to\infty} P_m(r)=-\infty$, so the theorem makes sense. In Theorem \ref{thm: radius}, generically, $g'(0)\neq 0$, and in this case the result simplifies.
\begin{cor}\label{cor: radius}
In the setting of the theorem above, assume that $g'(0)\neq0$. Then, for $$r>\frac{1+\vert g(0)\vert^2}{\vert g'(0)\vert},$$ $h$ is unstable.
\end{cor}
\begin{proof}
    Applying $\eqref{eq:polynomial_holomorphic_case}$ with $m=1$ yields the polynomial
    $$P_1(r)=-\vert b_k a_1\vert^2 r^{2}+\vert b_k\vert^2(1+\vert a_0\vert^2)^2.$$
    The only positive root of $P_1$ is $\frac{1+\vert a_0\vert^2}{\vert a_1\vert}$, and the result follows. 
\end{proof}
Theorem \ref{thm: radius} assumes that $g$ is holomorphic at $0$, but this is not important. Indeed, we can assume that $g$ does not surject onto $\C$, since then $h$ is already unstable by \cite{Sch}. Thus, using an orthogonal transformation of $\R^3$, one can make the Gauss map miss the north pole, so that $g$ becomes holomorphic. By Corollary \ref{cor: radius}, we obtain that for $g$ meromorphic with a simple pole at $0$ with residue $a_{-1}$, $h$ is unstable when $r>|a_{-1}|.$

Corollary \ref{cor: radius} can be reinterpreted in terms of the spherical area of $h$, i.e., the area of the image of the Gauss map in the round metric. For each $s<r$, we approximate the spherical area $\mathcal{A}(\cdot)$ of $h|_{\overline{\mathbb{D}}_s}$ as $$\mathcal{A}(h|_{\overline{\mathbb{D}}_s})=\int_{\mathbb{D}_s} \frac{4|g'(z)|^2}{(1+|g(z)|^2)^2} dxdy =\frac{4\pi|g'(0)|^2}{(1+|g(0)|^2)^2} s^2 + O(s^3).$$
Thus, writing $\mathcal{A}(h|_{\overline{\mathbb{D}}_s})=\mathcal{A}_2(h) s^2 + O(s^3)$, a restatement of Corollary \ref{cor: radius} is that if $$\mathcal{A}_2(h) r^2>4\pi,$$ then $h$ is unstable. From this standpoint, it can be seen as a partial converse of a theorem of Barbosa and Do Carmo \cite{BdC}, which says that if $\mathcal{A}(h)<2\pi$, then $h$ is stable (see also the generalization to $\R^n$ \cite{Carmo1980}).

\begin{remark}
    Another partial converse of the result from \cite{BdC} is an earlier result of Schwarz \cite{Sch}, which says that if the first Dirichlet eigenvalue of the image of the Gauss map is smaller than $2$, then $h$ is unstable. The connection is made clear by the standard corollary that if the Gauss map image contains a hemisphere (the spherical area is $2\pi$), then $h$ is unstable. We point out that Schwarz's result can be used to prove a weaker version of Corollary \ref{cor: radius}: assuming $g(0)=0$ for simplicity, by the Koebe $1/4$-Theorem, the Gauss map image contains a hemisphere when $r>\frac{4}{|g'(0)|}.$
\end{remark}

\subsection{Acknowledgments}
This paper is an outgrowth of Thomas-Ren\'e Thalmaier's bachelor's thesis (``m{\'e}moire") completed at the University of Luxembourg in the 2024-2025 academic year, under the supervision of Nathaniel Sagman.

\section{Preliminaries}

\subsection{Branched minimal immersions}\label{sec: basics}
Let $\Omega$ be a compact surface with $C^\infty$ boundary $\partial \Omega$ and let $\delta$ be the Euclidean metric on $\R^n$. One definition of a branched minimal immersion to $\R^n$ is a $C^\infty$ map $h:\Omega\to \R^n$ that extends continuously to $\overline{\Omega}$, and that is a critical point of the area functional 
\begin{equation}\label{eq: area functional}
    A(h(\Omega))=\int_{\Omega}\det h^*\delta,
\end{equation}
where variations are $C^\infty$ functions $N:\Omega\to \R^n$ that extend to $0$ on $\partial\Omega$. The image is called a branched minimal surface. A branched minimal immersion is said to be stable if for all variations $W$, 
\begin{equation}\label{eq: area unstable}
    \frac{d^2}{dt^2}|_{t=0}A(h+tW)\geq 0.
\end{equation}
Otherwise, we say it is unstable. If $\Omega$ is an open surface, a map $h:\Omega\to \R^n$ is a (stable) branched minimal immersion if when we restrict to any compact subsurfaces with $C^\infty$ boundary, we get a (stable) branched minimal immersion.

As in Remark \ref{rem: stability}, we emphasize we're considering variations of the map $h$, rather than variations of the surface $h(\Omega)$. For the latter, it would be natural to demand that $h$ is an isometric immersion, or that variations preserve the topology of $h(\Omega).$

For later use, we record the standard formula for the second variation of area. Let $\Omega$ be compact with $C^\infty$ boundary, and $h:\Omega\to \R^n$ a minimal immersion (note: we're assuming unbranched). Let $g=h^*\sigma$ and let $Nh(\Omega)$ be the normal bundle with the metric induced from $\delta$ on $\mathbb R^n$. For any normal variation $W$, i.e., $W$ inducing a section of $Nh(\Omega)$,
\begin{equation}\label{eq: second variation classical}
     \frac{d^2}{dt^2}|_{t=0}A(h+tW)
=
\int_\Omega
(
|\nabla^\perp W|_g^2
-
\sum_{i,j=1}^2 \langle \textrm{II}_h(e_i,e_j),W\rangle^2
)dV_g,
\end{equation}
where $\nabla^\perp$ is the normal connection, $|\cdot|_g$ is the norm induced by $g$ and $\delta$ on $T^*\Omega\otimes Nh(\Omega)$, $\textrm{II}_h$ is the second fundamental form, $(e_1,e_2)$ is any local $g$-orthonormal frame for the tangent space, $\langle \cdot,\cdot\rangle$ is Euclidean inner product on $\mathbb R^n$, and $dV_g$ is the volume form. When $h$ is branched, the formula (\ref{eq: second variation classical}) can be suitably reinterpreted to still hold (see \cite{Micallefnote}). 

\subsection{Weierstrass-Enneper data} 
Henceforth, for domains inside $\C$, we fix a global source coordinate $z$.
For a $C^1$ function $h$ on such a domain, we write $h_z=\frac{\partial h}{\partial z}$, and similar for $h_x$ and $h_y$. We often consider functions to $\R^n$, written in components as $(h_1,\dots, h_n)$; we write $h_{i,z}=\frac{\partial h_i}{\partial z}.$

Assume that $\Omega$ is a simply connected domain in $\C$, so that it inherits $z$. It is well known that $h:\Omega\to \R^n$ is a branched minimal immersion if and only if it is harmonic and conformal. Writing $h$ in components as $h=(h_1,\dots, h_n)$, harmonicity of $h_i$ is equivalent to the holomorphicity of the $1$-form $\alpha_i =h_{i,z} dz$, and conformality is equivalent to the relation 
\begin{equation}\label{eq:conformal}
    \sum_{i=1}^n \alpha_i^2 =0.
\end{equation}
Conversely, given any holomorphic $1$-forms $\alpha_i$ satisfying (\ref{eq:conformal}), integrating the real parts determines a branched minimal immersion $$h(z)=\Big (\int_{z_0}^z\Re\alpha_1,\dots,\int_{z_0}^z\Re\alpha_n\Big).$$
The data $(\alpha_1,\dots, \alpha_n)$ is called the Weierstrass-Enneper data of $h$. Evidently, this data determines $h$ up to translation.

When $\Omega$ is not simply connected, the same correspondence holds, except one has to restrict to $1$-forms with purely imaginary periods. Indeed, $h$ produces $(\alpha_1,\dots, \alpha_n)$ as before, and to go backwards one has to lift to the universal cover and then integrate the real parts. The period condition ensures that the branched minimal immersion descends from the universal cover back to $\Omega.$

We recall that a complex structure on an affine subspace of $\R^n$ is orthogonal if the almost complex structure preserves the Euclidean metric $\delta$, and we defined a map to $\R^n$ to be holomorphic up to a rigid motion if it lands in an affine subspace $A$ in which it is holomorphic for some orthogonal complex structure. If $J$ is an orthogonal complex structure on an affine subspace of $\R^n$, and $j$ is the ordinary complex structure on $\C$, the holomorphicity condition for a map $h$ is written as $dh\circ j =J\circ dh$. Splitting $h_z =h_x+ih_y$, this is equivalent to 
\begin{equation}\label{eq: holomorphicity with J}
    Jh_x=h_y, \hspace{1mm} Jh_y=-h_x.
\end{equation}
The theorem below explains how to read off holomorphicity for maps from $\C\to \R^n$ from the Weierstrass-Enneper data. Let $g:\C^n\times \C^n\to\C$ be the standard bilinear form $$g(u,v)=\sum_{i=1}^n u_iv_i.$$ Given a branched minimal immersion $h=(h_1,\dots, h_n):\C\to \R^n$, we can write $h_{i,z}=\sum_{j=0}^\infty c_j^i z^j$ and note that $h$ is equivalent (up to translations) to the sequence of vectors $(c_j)_{j=0}^\infty\in \C^n$, where $c_j=(c_j^1,\dots, c_j^n).$
\begin{thm}
\label{thm:main2}
    The map $h$ is holomorphic up to a rigid motion if and only if for all $m,k,$ $g(c_m,c_k)=0.$
\end{thm}
The geometric content of Theorem \ref{thm:main2} is that $h$ is holomorphic up to a rigid motion if and only if every holomorphic $k$-jet $\frac{1}{(k-1)!}\partial_z^k h(0)$ lies in the same isotropic subspace of $\C^n$. The theorem follows from a general lemma. Given a complex subspace $L$ of $\C^n$, we set $\Re(L)\subset \R^n$ to be the image of $L$ under the real part map $u\mapsto \Re (u)$, and $\Im (L)\subset \R^n$ to be the image under the imaginary part map.
\begin{lem}\label{lem: isotropy}
    Let $L\subset \C^n$ be a complex subspace and $W=\Re(L)\oplus \Im (L) \subset \R^n$. Then $L$ is $g$-isotropic if and only if $W$ carries an orthogonal complex structure $J$ such that $L$ is the $i$-eigenspace.
\end{lem}
\begin{proof}
If $W$ admits such a complex structure $J$, the eigenspace condition can be written explicitly as $L=\{X-iJX: X\in W\}$. Expressing
$$g(X-iJX,Y-iJY) = g(X,Y) -g(JX,JY)-i(g(JX,Y)+g(X,JY)),$$ orthogonality gives $g$-isotropy: $g(JX,JY)=g(X,Y)$, $g(JX,Y)=g(J^2X,JY)=-g(X,JY).$

Conversely, assuming $L$ is $g$-isotropic, we claim that the real part map $L\to\Re(L)$ is an isomorphism. Indeed, surjectivity is obvious, and for injectivity, note that the kernel is equal to $i\Im (L)\cap L.$ For $iv\in i\Im (L)\cap L$, $$\delta(v,v) = -g(iv,iv)=0.$$ Positive-definiteness of $\delta$ forces $v=0,$ which establishes injectivity. From the claim, we have that for all $u\in W$, there exists a unique $x=u+iv\in L$ such that $\Re(x)=u$. We define $J:W\to W$ by $J(u)=-v$. Since $ix\in L$, it is clear that $J^2=-I$. It's also easy to see from the definition that $L$ is the $i$-eigenspace. Writing $$\Re g(X-iJX,Y-iJY) = g(X,Y) -g(JX,JY),$$ $g$-isotropy clearly implies orthogonality. 
\end{proof}
\begin{proof}[Proof of Theorem \ref{thm:main2}]
Assume $h$ lives in an even dimensional affine subspace $W+b\subset \R^n$ and is holomorphic for some orthogonal complex structure on $W+b$. Here, $W$ is a real vector subspace and $b\in \R^n$. Then, $h_z(z)\in W$ for all $z$, and \eqref{eq: holomorphicity with J} shows that $h_z(z)$ lies in the $i$-eigenspace $L$ of the almost complex structure. By successively differentiating in $z$ and evaluating at $z=0$, every $c_n$ lies in $L$. Thus, by Lemma \ref{lem: isotropy}, $g(c_m,c_k)=0$ for every $m,k$.

Conversely, if $g(c_m,c_k)=0$ for every $m,k$, we  apply Lemma \ref{lem: isotropy} to the subspace $L=\textrm{Span}_{\C}(c_m: m=0,1,\dots\}.$ Then $W=\textrm{Span}(a_m,b_m: m=0,1,\dots)$ carries an orthogonal complex structure such that $L$ is the $i$-eigenspace. Observe that $h_z(z)\in L$ for every $z$. By integrating $\Re h_z$, we see that $h$ lives in some affine subspace $W+b\subset \R^n$ that inherits an orthogonal almost complex structure $J$, which is integrable since it is constant in coordinates. The $i$-eigenspace condition for $h_z(z)$ is equivalent to (\ref{eq: holomorphicity with J}), i.e., holomorphicity, for $h$.
\end{proof}

\section{Proof of main Theorems}

\subsection{The destabilization scheme}\label{sec: destabilizing}
To prove Theorem \ref{thm: main_C_case}, we use the destabilization scheme proposed by the first named author and Markovi{\'c} in \cite{MS}. We first state the main theorem we use. Afterward, mostly to keep the paper self-contained, we'll explain where the theorem came from. 

In the theorem below, we refer to the vector space $\mathcal{V}$ of first order variations of quasiconformal homeomorphisms of the disk (the tangent space at the identity of the universal Teichm{\"u}ller space). All we need to know about $\mathcal{V}$ (see \cite[\S 4.1]{MS} for the formal definition) is that it consists of holomorphic functions on $\mathbb{C}\backslash\mathbb{D}$ that vanish at $\infty$ and, for every $2<p<\infty$, extend to (not holomorphic) $L^p$ functions on all of $\C$ with certain properties. It is proved in \cite[\S 4.1]{MS} that $\mathcal{V}$ contains all holomorphic functions on $\mathbb{C}\backslash\mathbb{D}$ that vanish at $\infty$ and extend smoothly to $\mathbb{C}$. The prototypical example of $\varphi\in \mathcal{V}$ is $\varphi(z)=z^{-m}$, $m\geq 1$. For the theorem, we also define the functional $\mathcal{F}: C^1(\mathbb{D})\to \R$, 
\begin{equation}\label{eq: functional}
    \mathcal{F}(f) = \Re\int_{\mathbb{D}}f_zf_{\overline{z}}dxdy+\int_{\mathbb{D}}|f_{\overline{z}}|^2dxdy,
\end{equation}
originally from \cite[Section 5]{M2}. 
\begin{thm}[Theorem C in \cite{MS}]\label{thm: thmc}
    Let $h=(h_1,\dots, h_n):\overline{\mathbb{D}}\to \R^n$ be a branched minimal immersion such that no $h_i$ has a zero on $\partial\mathbb{D}$. For $\varphi\in \mathcal{V}$. For each $i$, let $v_i$ be the harmonic extension of $h_{i,z}\cdot \varphi|_{\partial \mathbb{D}}:\partial \mathbb{D}\to\mathbb{C}$. If 
    \begin{equation}\label{eq: instability equation}
        \mathcal{F}_h(\varphi) := \sum_{i=1}^n \mathcal{F}(v_i)<0,
    \end{equation}
then $h$ is unstable.
\end{thm}
Note that it is proved in \cite[Theorem B]{MS} that if a branched minimal immersion is unstable, then there exists $\varphi\in \mathcal{V}$ such that (\ref{eq: instability equation}) holds, and that the ordinary Morse index for the area is equal to the corresponding index for $\mathcal{F}_h$ on $\mathcal{V}$.

Theorem \ref{thm: thmc} comes from studying stability through the lens of energy rather than area. Given a branched minimal immersion, $h:\overline{\mathbb{D}}\to \mathbb{C}$,
let $$\mathcal{E}(\mathbb{D},h)=\int_{\Omega}|dh|^2 dx dy$$ be the Dirichlet energy. Suppose one is given an $n$-tuple of paths of maps $t\mapsto f_i^t:\mathbb{C}\to \C$, $i=1,\dots, n$, starting at the identity, all fixing the origin, and agreeing on $\mathbb{C}\backslash \mathbb{D}$ with a holomorphic map $F^t$. Observe that for all $t$ and $i,j$, $f_i^t(\mathbb{D})=f_j^t(\mathbb{D})$, and the image of every $$h^t:=(h_1\circ (f_1^t)^{-1},\dots, h_n\circ (f_n^t)^{-1}):f_1^t(\mathbb{D})\to \R^n$$ has the same boundary curve as $h$. Suppose that
\begin{equation}\label{eq: energy unstable}
    \frac{d^2}{dt^2}|_{t=0}\mathcal{E}(f_1^t(\mathbb{D}), h^t)<0.
\end{equation}
Then, since energy dominates area, one can reparametrize $f_1^t(\mathbb{D})$ so that $h^t$ is identified with a map of the form $h+tN+o(t)$ such that (\ref{eq: area unstable}) fails:
$$  \frac{d^2}{dt^2}|_{t=0}A(f+tN)< 0,$$
i.e., $N$ makes $h$ unstable (see \cite[\S 4.1]{MS} for details). We refer to a path of the form $t\mapsto h^t$ as a self-maps variation. 

As in \cite{MS}, one writes
$$f_i^t(z) = z+t\varphi_i(z) + o(t),$$ for some function $\varphi_i$ that globally extends a function $\varphi\in \mathcal{V}$. Note that for all $i$, $\varphi_i|_{\mathbb{C}\backslash\mathbb{D}}=\varphi$. Finding a destabilization of a branched minimal immersion then reduces to picking $\varphi_1,\dots, \varphi_n$ such that $t\mapsto f^{t\varphi_i}$ leads to (\ref{eq: energy unstable}). In \cite{MS}, the authors compute the second variation of energy in terms of $\varphi_i$'s, which leads to the functional $\mathcal{F}$, and the harmonic extensions considered in the statement of Theorem \ref{thm: thmc} come from taking the optimal choices.

While the route to destabilizing minimal surfaces is less conventional, the advantage of using (\ref{eq: instability equation}) over (\ref{eq: area unstable}) is that it's very tractable computationally. Theorem D and Corollary D from \cite{MS}, and now Theorem \ref{thm: main_C_case}, give a proof of concept.

\subsection{Proof of Theorem \ref{thm: main_C_case}}\label{sec: C case}
To put ourselves in the context of Theorem \ref{thm: thmc}, we note that a branched minimal immersion from $\C$ to $\R^n$ is stable if and only if its restriction to every $\overline{\mathbb{D}}_r=\{z\in \C: |z|\leq r\}$ is stable, and we'll replace maps from $\overline{\mathbb{D}}_r$ to $\R^n$ with maps from $\overline{\mathbb{D}}$ to $\R^n$ by conformally rescaling the source. That is, if $h: \overline{\mathbb{D}}_r \to\R^n$ is minimal, then we'll instead consider $h^r:\overline{\mathbb{D}} \to\R^n$, $h^r(z)=h(zr)$ (not to be confused with variations $t\mapsto h^t)$. We can then say that $h:\mathbb{C}\to \R^n$ is stable if and only if every $h^r:\overline{\mathbb{D}} \to\R^n$ is stable. 

The ``if" direction of Theorem \ref{thm: main_C_case}, that a holomorphic map is stable, is classical. For completeness, and as a warm-up, we give a new proof using Theorem \ref{thm: thmc}.  Given a function $g$ on the circle, we denote the harmonic (Poisson) extension by $P(g).$ In the proof below, and in the rest of this section, to make our expressions cleaner, we omit ``$dxdy$" from our integrals, with the area form $dx\wedge dy$ implicitly understood.

\begin{proof}[Proof of Theorem \ref{thm: main_C_case}, ``if" direction]
    Because we can conformally rescale $\mathbb{D}_r$ to $\mathbb{D}$, it suffices to show that any holomorphic up to a rigid motion $h=(h_1,\dots,h_n):\overline{\mathbb{D}}\to \R^n$ defines a stable branched minimal immersion. By Theorem \ref{thm:main2}, writing $h_{i,z}=\sum_{j=0}^\infty c_j^i z^m$ and $c_j=(c_j^1,\dots, c_j^n)$, we have $g(c_m,c_k)=0$ for all $m,k\geq 0$ ($g$ as in Theorem \ref{thm:main2}). By Theorems B and C from \cite{MS}, stability is equivalent to the assertion that for all test functions $\varphi\in\mathcal{V}$, $$\mathcal{F}_h(\varphi) = \sum_{i=1}^n \mathcal{F}(P(\varphi \frac{\partial h_i}{\partial z})) \geq 0,$$ where we recall that $\mathcal{F}_h$ is defined by (\ref{eq: functional}). For any $\varphi$ and for each $i=1,\dots, n$, we calculate the first term from (\ref{eq: functional}) using linearity of the harmonic extension: 
$$\int_\mathbb{D} \partial_zP(\varphi h_{i,z})\,\partial_{\overline{z}}P(\varphi h_{i,z})=\sum_{m\in\mathbb{Z}}\sum_{k\in\mathbb{Z}}c^i_mc^i_k\int_\mathbb{D}\partial_zP(\varphi z^m)\, \partial_{\overline{z}}P(\varphi z^k).$$ Hence, summing over $i$,
we obtain
$$\sum_{i=1}^n\int_\mathbb{D} \partial_zP(\varphi h_{i,z})\,\partial_{\overline{z}}P(\varphi h_{i,z})=\sum_{k=0}^\infty\sum_{m=0}^\infty g(c_m,c_k)\int_\mathbb{D}\partial_zP(\varphi z^m)\, \partial_{\overline{z}}P(\varphi z^k),$$
which is equal to zero by the condition $g(c_m,c_k)=0$. Thus, $$\mathcal{F}_h(\varphi)=\sum_{i=1}^n \int_{\mathbb{D}} |(P(\varphi \frac{\partial h_i}{\partial z}))_{\overline{z}}|^2\geq 0.$$ By the discussion above, $h$ is stable. 
\end{proof}

The main computation behind the proof of the ``only if" direction is contained in the lemma below.

\begin{lemma}\label{lem: main computation}
    Let $h:\C\to \R$ be a harmonic function with $h_z(z)=\sum_{j=0}^\infty c_j z^j$ and let $h^r$ be as above. For positive integers $k,m$ and $\theta\in [0,2\pi]$ and $y\in\R$, define $\varphi=\varphi_{k,m,\theta,y}\in\mathcal{V}$ by $\varphi_{k,m,\theta,y}(z)=z^{-k}+ye^{i\theta}z^{-m}$. Then,
    \begin{align*}
        \int_{\mathbb{D}} |(P(h^r\varphi))_{\overline{z}}|^2&=\pi\sum_{j=0}^{k-1}(k-j)\vert c_j\vert^2 r^{2j}+\pi y\sum_{j=0}^{m-1}(m-j)\Re(c_j\overline{c_{j+k-m}})r^{2j+k-m}\\
        &+\pi y^2\sum_{j=0}^{m-1}(m-j)\vert c_j\vert^2r^{2j}
    \end{align*}
 and 
    \begin{align*}
        \Re \int_{\mathbb{D}} (P(h^r\varphi))_{z}(P(h^r\varphi))_{\overline{z}}&=\pi y^2\sum_{j=0}^{m-1}(m-j)\Re(c_jc_{2m-j}e^{2i\theta})r^{2m}+ \pi y\sum_{j=0}^{m-1}(m-j)\Re(c_{j}c_{k+m-j}e^{i\theta})r^{k+m}\\
        &+ \pi y \sum_{j=0}^{k-1}(k-j)\Re(c_{j}c_{k+m-j}e^{i\theta})r^{k+m} + \pi\sum_{j=0}^{k-1}(k-j)\Re(c_jc_{2k-j})r^{2k}.
    \end{align*}
\end{lemma}
\begin{proof}
    Writing,
$$\varphi(z) h_{z}^r(z)=\sum^\infty_{j=0}c_jr^jz^{j-k}+ye^{i\theta}\sum^\infty_{j=0}c_jr^jz^{j-m}$$
the harmonic extension of the restriction to $\partial\mathbb{D}$ is
$$P(\varphi h_z^r)(z)= \sum_{j=0}^{k}c_jr^j\overline{z}^{k-n}+\sum_{j=k+1}^{\infty}c_j r^j z^{j-k}+ye^{i\theta}\sum_{j=0}^{m}c_j r^j \overline{z}^{m-j}+ye^{i\theta}\sum_{j=m+1}^{\infty}c_j r^j z^{j-m}.$$
Differentiating, we obtain
     \[
    \begin{aligned}
    P(\varphi h_{z}^r)_{\overline z}(z)
    &=
    \underbrace{\sum_{j=0}^{k-1} c_j r^j (k-j)\,\overline z^{\,k-j-1}}_{A:=}
    \;+\;
    \underbrace{ye^{i\theta}\sum_{j=0}^{m-1} c_j r^j (m-j)\,\overline z^{\,m-j-1}}_{B:=},
    \\[1ex]
    P(\varphi h_{z}^r)_{z}(z)
    &=
    \underbrace{\sum_{j=k+1}^{\infty} c_j r^j(j-k)\,z^{\,j-k-1}}_{C:=}
    \;+\;
    \underbrace{y e^{i\theta}\sum_{j=m+1}^{\infty} c_j r^j (j-m)\,z^{\,j-m-1}}_{D:=}.
    \end{aligned}
    \]
To calculate integrals below, we will use the orthogonality property
\begin{equation}
\label{orthogonality}
    \int_\mathbb{D}z^a\overline{z}^b=\int_0^{2\pi}\int_0^1s^{a+b+1}e^{i\theta(a-b)}dsd\theta=\frac{\pi}{a+b+2}\mathbf{1}_{a=b}
\end{equation}
several times.
We first compute $\int_\mathbb{D}\vert P(\varphi h_{z}^r)_{\overline z}\vert^2$. Writing $\vert P(\varphi h_{z}^r)_{\overline z}\vert^2=A\overline{A}+2\Re A\overline{B}+B\overline{B}$, we go term by term. Using $\eqref{orthogonality}$,
\begin{align*}
    \int_\mathbb{D}A\overline{A}&=\sum_{j_1,j_2=0}^{k-1}c_{j_1}\overline{c_{j_2}}r^{j_1+j_2}(k-j_1)(k-j_2)\int_\mathbb{D}\overline{z}^{k-j_1-1}z^{k-j_2-1}\\
    &=2\pi\sum_{j_1,j_2=0}^{k-1}c_{j_1}\overline{c_{j_2}}r^{j_1+j_2}(k-j_1)(k-j_2)\frac{1}{2k-j_1-j_2} \mathbf{1}_{j_1=j_2}\\
    &=\pi\sum_{j=0}^{k-1}(k-j)\vert c_j\vert^2r^{2j}.
\end{align*}
Analogous computations give
$$\int_\mathbb{D}B\overline{B}=\pi y^2\sum_{j=0}^{m-1}(m-j)\vert c_j\vert^2r^{2j}$$
and 
$$\int_\mathbb{D}A\overline{B}= \pi y\sum_{j=0}^{m-1}(m-j)c_j\overline{c_{j+k-m}}r^{2j+k-m}.$$
We've done all of the computations that give the first expression in the statement of the lemma. We proceed in a similar way for $\Re\int_\mathbb{D}P(\varphi h_{z}^r)_zP(\varphi h_{z^r})_{\overline{z}}$ using $P(\varphi h_{z}^r)_zP(\varphi h_{z}^r)_{\overline{z}}=AC+AD+BC+BD$. As above, we compute directly using $\eqref{orthogonality}$. We show the details for $\Re AD$ and $\Re BC$, and leave it to the reader to check the simpler computations that give
$$\Re\int_\mathbb{D}BD=\pi y^2\sum_{j=0}^{m-1}(m-j)\Re(c_jc_{2m-j}e^{2i\theta})r^{2m}$$
and
$$\Re\int_\mathbb{D}AC=\pi\sum_{j=0}^{k-1}(k-j)\Re(c_jc_{2k-j})r^{2k}.$$
For the real part of the integral of $AD$, we compute
\begin{align*}
        \Re\int_\mathbb{D}AD &=y\Re\sum_{j_1=0}^{k-1}\sum_{j_2=m+1}^{\infty} (k-j_1)(j_2-m)c_{j_1}c_{j_2}r^{j_1+j_2} e^{i\theta}\int_\mathbb{D} \overline{z}^{k-j_1-1} z^{j_2-m-1}\\
        &=2\pi y\sum_{j_1=0}^{k-1}\sum_{j_2=m+1}^{\infty}\frac{(k-j_1)(j_2-m)}{k-m-j_1+j_2} \Re(c_{j_1}c_{j_2}e^{i\theta})r^{j_1+j_2} \mathbf{1}_{j_2=k+m-j_1}\\
        &=\pi y \sum_{j=0}^{k-1}(k-j)\Re(c_{j}c_{k+m-j}e^{i\theta})r^{k+m}.
\end{align*}
Similarly, for the real part of the integral of $BD$,
\begin{align*}
    \Re\int_\mathbb{D}BC &=y\sum_{j_1=0}^{m-1}\sum_{j_2=k+1}^{\infty}(m-j_1)(j_2-k)\Re(c_{j_1}c_{j_2}e^{i\theta})r^{j_1+j_2}\int_\mathbb{D}\overline{z}^{m-j_1-1}z^{j_2-k-1}\\
    &=2\pi y\sum_{j_1=0}^{m-1}\sum_{j_2=k+1}^{\infty}\frac{(m-j_1)(j_2-k)}{m-k+j_2-j_1}\Re(c_{j_1}c_{j_2}e^{i\theta})r^{j_1+j_2}\mathbf{1}_{j_2=k+m-j_1}\\
    &=\pi y\sum_{j=0}^{m-1}(m-j)\Re(c_{j}c_{k+m-j}e^{i\theta})r^{k+m}.
\end{align*}
Putting everything together, the proof is complete.
\end{proof}
\begin{lem}\label{lem: E_k is zero}
    Let $N>0$ be an integer and let
     $a=(a_j)_{j=1}^N\in \C^N$ be a vector such that for all $j$, $a_{j}=a_{N-j}$. For all $0\leq k \leq \left\lfloor \frac{N}{2} \right\rfloor$, set $$E_k=\sum_{j=0}^{k-1}(k-j)a_j+\sum_{j=0}^{N-k-1}(N-k-j)a_j.$$ If $E_k=0$ for every $k$, then $a=0.$
\end{lem}
 Writing $E_k=E_k^1+E_k^2$, where $E_k^1=\sum_{j=0}^{k-1}(k-j)a_j$ and $E_k^2=\sum_{j=0}^{N-k-1}(N-k-j)a_j$, note that for $k=0$, the expression for $E_k^1$ is interpreted as $0$.
\begin{proof}
    Observe that for all $k\geq 1$, $$E_{k-1}^1-E_{k}^1 = -\sum_{j=0}^{k-1}a_j, \hspace{1mm} E_{k-1}^2-E_{k}^2=\sum_{j=0}^{N-k-1} a_j.$$ Indeed, $$E_{k-1}^1-E_{k}^1= \Big (\sum_{j=0}^{k-2} ((k-1-j)-(k-j))a_j\Big ) -a_{k-1}= -\sum_{j=0}^{k-1}a_j,$$ and similar for $E_{k-1}^2-E_{k}^2$. We deduce, using that $E_k=0$ for all $1\leq k\leq \left\lfloor \frac{N}{2} \right\rfloor$, 
    \begin{equation}\label{eq: sum is 0}
        \sum_{j=k}^{N-k-1} a_j=E_{k-1}-E_{k}=0.
    \end{equation}
    Setting $\ell=\left\lfloor \frac{N}{2} \right\rfloor,$ we now use (\ref{eq: sum is 0}) to prove by induction that for $p=0,\dots, \ell-1$, we have $a_{\ell-p}=0.$ For the base case $p=0$, we use \eqref{eq: sum is 0} with $k=\ell$. If $N=2\ell$ is even, the sum above is just $a_\ell$, and hence $a_\ell=0$. If $N=2\ell+1$ is odd, \eqref{eq: sum is 0} becomes $a_{\ell}+a_{\ell+1}=0$, and then we deduce $a_\ell=0$ by symmetry. For the induction step, assume that for all $0\leq q < p$ we have $a_{\ell-q}=0$. Using \eqref{eq: sum is 0} with $k=\ell-p$, \eqref{eq: sum is 0} returns $$0=a_{\ell-p}+\Big(\sum_{j=\ell-p+1}^{N-\ell+p-1} a_j\Big )+ a_{N-\ell+p}= 2a_{\ell-p}+\Big(\sum_{j=\ell-p+1}^{N-\ell+p-1} a_j\Big ).$$ By the induction hypothesis, all terms in the sum on the right vanish, and therefore $a_{\ell-p}=0$.

    The now complete induction, together with symmetry, showed that $a_j=0$ for all $j\neq 0,N$. To see $a_0=a_N=0$, just note that for any $k$, $E_k=0$ expresses $a_0=a_N$ as a linear combination of other $a_j$'s, which we know are all zero. This establishes the result.
\end{proof}

\begin{proof}[Proof of Theorem \ref{thm: main_C_case}, ``only if" direction]
Let $h=(h_1,\dots, h_n):\C\to \R^n$ be a stable branched minimal immersion. As in the set-up for Theorem \ref{thm:main2}, write $h_{i,z}=\sum_{j=0}^\infty c_j^i z^m$, and $c_j=(c_j^1,\dots, c_j^m)$. We show by induction on $N=k+m$ that $g(c_k,c_m)=0$ for any $k,m\geq0$, and then the result will follow from Theorem \ref{thm:main2}. 

By minimality, the sum of $h_{i,z}^2$ over $i=1,\dots,n$ is zero. By evaluating at $z=0$, we deduce that $\sum_{i=1}^n (c^i_0)^2=0$. It follows that $g(c_0,c_0)=0$.

Let $N>0$ and assume that $g(c_k,c_m)=0$ for any $0\leq k+m<N$. We aim to show that $g(c_j,c_{N-j})=0$ for any $j=0,\dots,N$. By Lemma \ref{lem: E_k is zero}, it reduces to showing that for all $0\leq k\leq\left\lfloor \frac{N}{2} \right\rfloor$, $E_k^N=0$, where $$E_k^N:=\sum_{j=0}^{k-1}(k-j)g(c_{j},c_{N-n})+\sum_{j=0}^{N-k-1}(N-k-j)g(c_{j},c_{N-j}).$$ Indeed, knowing that $E_k^N=0$, we can set $a_j = g(c_{j},c_{N-n})$, and then Lemma \ref{lem: E_k is zero} yields the result (here, $E_k^N$ is the ``$E_k$" from Lemma \ref{lem: E_k is zero}).

Fix $k$ between $0$ and $\left\lfloor \frac{N}{2} \right\rfloor$. For $m:=N-k$, we apply Theorem $\ref{thm: thmc}$ to $h^r:\overline{\mathbb{D}}\to \R^n$, $h^r(z)=h(rz)$, for every $r>0$ such that no $h_{i,z}$ has a zero on $\partial\mathbb{D}_r$ (that is, every $r>0$ except a discrete subset), using the destabilization function $\varphi_{k,m,\theta,y}(z):=z^{-k}+ye^{i\theta}z^{-m}$, where we'll choose $\theta\in[0,2\pi]$ and $y\in\mathbb{R}$ later. Lemma \ref{lem: main computation} gives the expressions that go into $\mathcal{F}_{h^r}(\varphi_{k,m,\theta,y})$. Firstly, using Lemma \ref{lem: main computation}, we can write $$\sum_{i=1}^n \int_{\mathbb{D}} |(P(\varphi_{k,m,\theta,y} h_{i,z}^r))_{\overline{z}}|^2 = y^2P_1(r) + yQ_1(r)+T_1(r),$$ where $P_1$, $Q_1$ and $T_1$ are polynomials in $r$ of degrees less or equal to $2m-2$, $k+m-2$ and $2k-2$ respectively. For $\Re\int_\mathbb{D}P(\varphi_j h_{i,z})_zP(\varphi_j h_{i,z})_{\overline{z}}$, if we sum up the real parts of the integrals of the ``AC terms" from the proof of Lemma \ref{lem: main computation} (the final sum in the final expression), we obtain $$\sum_{i=1}^n \pi\sum_{j=0}^{k-1}(k-j)\Re(c_jc_{2k-j})r^{2k}=\pi\sum_{j=0}^{k-1}(k-j)\Re(g(c_j,c_{2k-j}))r^{2k}.$$
Since $2k<N$, every $g(c_j,c_{2k-j})$ is zero. Hence, the expression above is zero, and these terms don't figure into the final expression for $\mathcal{F}_{h^r}(\varphi_{k,m,\theta,y}).$ Next, we sum the ``AD" and ``BC" terms from Lemma \ref{lem: main computation}, and we see $E_k^N$:
$$
\sum_{i=1}^n \pi y \Big (\sum_{j=0}^{m-1}(m-j)\Re (c_j^i c_{k+m-j}^ie^{i\theta}) + \sum_{j=0}^{k-1}\Re (c_j^ic_{k+m-j}^i e^{i\theta})\Big ) r^{k+m}=\pi y\Re(E_k^Ne^{i\theta})r^{k+m}.$$

With the expressions above in hand, assume for the sake of contradiction that $E_k^N\neq 0$. Then, we can choose $\theta$ such that $\Re(E_k^Ne^{i\theta})>0$. Putting everything together,
$$\Re\int_\mathbb{D} P(\varphi_{k,m,\theta,y} h_{i,z}^r)_zP(\varphi_{k,m,\theta,y} h_{i,z}^r)_{\overline{z}}=y^2P_2(r)+yQ_2(r),$$
where $P_2$ is a polynomial of degree less then or equal to $2m$ and $Q_2$ is a polynomial of degree $m+k$. Defining $P:=P_1+P_2$, $Q:=Q_1+Q_2$, and $T:=T_1$, we have
$$\mathcal{F}_{h_2}(\varphi_{k,m,\theta,y})=y^2P(r)+yQ(r)+T(r),$$
where the degrees of $P$ and $T$ are less than or equal to $2m$ and $2k-2$ respectively, and $Q$ has degree $m+k$. Therefore, viewing $\mathcal{F}_{h_2}(\varphi_{k,m,\theta,y})$ as a quadratic polynomial in $y$, we can compute the discriminant
$$\Delta_r=Q(r)^2-4P(r)T(r),$$
which is a polynomial in $r$ of degree $2m+2k$ with a strictly positive leading term. Necessarily, $\Delta_r$ tends to infinity as $r$ grows, and we can choose $r_0>0$ such that $\Delta_{r_0}>0$. Consequently, the polynomial in $y$ has two distinct roots and we can conclude that there exists $y_0\in\mathbb{R}$ such that
$$y_0^2P(r_0)+y_0Q(r_0)+T(r_0)<0.$$
Choosing $y=y_0$ and $\theta$ as above, $\varphi_{k,m,\theta,y_0}$ destabilizes $h_{r_0}$. Since $h$ was stable, we have a contradiction. The conclusion is that $E_k^N=0$ for any $0\leq k\leq\left\lfloor \frac{N}{2} \right\rfloor$. As discussed above, an application of Lemma \ref{lem: E_k is zero} completes the proof.
\end{proof}

\subsection{Proof of Theorem \ref{thm: main}}\label{sec: reducing}
Using Theorem \ref{thm: main_C_case}, we now prove Theorem \ref{thm: main}, along with a slight extension, Theorem \ref{thm: main_extended}. We first prove a lemma. Let $\Omega\subset \C$ be a bounded domain with $C^\infty$ boundary and let $h:\Omega\to \R^n$ a branched minimal immersion with Weierstrass-Enneper data $(\alpha_1,\dots, \alpha_n).$
For any holomorphic function $f:\Omega\to \C$, we consider the new family of holomorphic $1$-forms $(f\alpha_1,\dots, f\alpha_n)$. The condition (\ref{eq:conformal}) holds for $(f\alpha_1,\dots, f\alpha_n)$, and provided each $f\alpha_i$ has purely imaginary periods, this collection determines a branched minimal immersion $h_f:\Omega\to \R^n$, unique up to translation. For convenience, we assume that $h$ and $h_f$ agree at a point.
\begin{lem}\label{lem: properties preserved}
    Let $\Omega$, $h$, $(\alpha_1,\dots, \alpha_n)$, and $f$ be as above, and assume that $h_f$ is well-defined.
    \begin{enumerate}
        \item $h$ is stable if and only if $h_f$ is stable. More generally, the area functionals for $h$ and $h_f$ have the same Morse index.
        \item Let $J$ be an orthogonal complex structure on an even dimensional affine subspace of $\R^n$. $h$ is holomorphic for $J$ if and only if $h_f$ is holomorphic for $J$.
    \end{enumerate}
\end{lem}
The intuition for the lemma is that if we conformally reparametrize a minimal surface, both stability and holomorphicity are preserved. Away from branch points of $f$, $h_f$ is locally a conformal reparametrization of $h$.

Toward the proof, recall that every branched minimal immersion $F:\Omega\to \R^n$ comes with a generalized Gauss map $G_F:\Omega\to \mathbb{CP}^{n-1}$, $G_F(z)=[\frac{\partial F_1}{\partial z}(z),\dots, \frac{\partial F_n}{\partial z}(z)]$. The map $G_F$ is holomorphic and determines the normal bundle (a priori defined where $h$ is an immersion). Although we won't really need this below, the map $G_F$ allows one to extend the normal bundle over the branch locus. 
\begin{proof}
     The key point is that the generalized Gauss maps $G_h,G_{h_f}: \Omega\to \mathbb{CP}^{n-1}$ are equal. It follows that the (extended) normal bundles of $h$ and $h_f$, viewed as sitting inside $T\R^n$, are isomorphic via the map $\iota$ defined by $(h_f(p),v)\mapsto (h(p),v).$

     We begin with (1). By the log cut-off trick (see \cite{Micallefnote} and \cite[\S 4.4]{MarkovicSagmanSmillie2025}), we're welcome to restrict to normal variations defined on the complement of sufficiently small disks around the branch points of both $h$ and $h_f$. We can then use the formula (\ref{eq: second variation classical}). By the comments above, normal variations for $h$ are normal for $h_f$ and vice versa. Now we just observe that if we input any normal variation $W$, $$\frac{d^2}{dt^2}|_{t=0}A(h+tW) = \frac{d^2}{dt^2}|_{t=0}A(h_f+tW).$$ Indeed, if $h^*\delta=g$, then $h_f^*\delta = |f|^2 g$ and the normal projections for both maps are equal. We then routinely calculate $dV_{|f|^2g}=|f|^2 dV_g$, $|\nabla^\perp W|_{|f|^2g}^2=|f|^{-2}|\nabla^\perp W|_g^2$, and $\langle \textrm{II}_{h}(e_i,e_j),W\rangle^2= |f|^{-2}\langle \textrm{II}_{h_f}(e_i,e_j),W\rangle^2$, and see that the integrands in (\ref{eq: second variation classical}) are equal. The result follows.

     Item (2) follows immediately from \cite[Proposition 1.6]{Law}, but also it's very easy to see directly. Since $h$ and $h_f$ agree at a point and have the same generalized Gauss map, it follows that one lands in an even dimensional affine subspace if and only if the other does. We write $h_z=h_x+ih_y$ and $f=u+iv$, so that $$(h_f)_z=(uh_x-vh_y)+i(uh_y+vh_x).$$ Then it is easily checked that (\ref{eq: holomorphicity with J}) holds for $h$ if and only if it holds for $h_f$. 
\end{proof}

We now prove Theorem \ref{thm: main}. We'll use Lemma \ref{lem: properties preserved} to cancel off singularities of the Weierstrass-Enneper data, which, modulo a few technical points, reduces Theorem \ref{thm: main} to the case of a map from $\C$ to $\R^n.$
\begin{proof}[Proof of Theorem \ref{thm: main}]
    It is well known that holomorphic implies stable (see also the top of \S \ref{sec: C case}). Let $\Sigma=\C\backslash P$ and $h:\Sigma\to \R^n$ be as in the statement of the theorem, and assume that $h$ is stable. Since $h(\Sigma)$ is locally complete with finite total curvature at all ends except the one corresponding to $\infty$, the Weierstrass-Enneper data $(\alpha_1,\dots,\alpha_n)$ consists of meromorphic functions with poles at $P$ (but there might be an essential singularity at $\infty$). Indeed, a minor extension of \cite[Theorem 1]{CO} shows that $h$ is locally algebraic around the punctures; for a source, the well-known argument is contained in the proof of Theorem 4.1.1 in \cite{AlarconLopezForstneric2021}. (It follows that $P$ is discrete, as was mentioned in the introduction.)

    Now, one of the first consequences of the Weierstrass factorization theorem is that any meromorphic function $g$ on $\C$ can be written $g=g^+/g^-$, where $g^+$ and $g^-$ are holomorphic. For each $\alpha_i$, we write $\alpha_i(z) dz = \frac{\alpha_i^+(z)}{\alpha_i^-(z)}dz$. Set $$f=\prod_{i=1}^n \alpha_i^-$$ and consider the tuple of meromorphic $1$-forms $(f\alpha_1,\dots, f\alpha_n)$. By construction, each $f\alpha_i$ extends to an entire function on $\C$. Hence, with no period problem, we can integrate the real parts to obtain a branched minimal immersion $h_f:\C\to \R^n$ with Weierstrass-Enneper data $(f\alpha_1,\dots, f\alpha_n)$. By item (1) in Lemma \ref{lem: properties preserved}, the restriction of $h_f$ to any open subset of $\C\backslash P$, is stable. By the log cut-off trick (as above, see \cite{Micallefnote}, \cite[\S 4.4]{MarkovicSagmanSmillie2025}), we can perturb any hypothetical destabilizing variation of $h$ that's compactly supported on a disk $\mathbb{D}_R$ to a destabilizing variation compactly supported on $\mathbb{D}_R\backslash (\mathbb{D}_R\cap P).$ It follows from stability on $\C\backslash P$ that $h_f$ is stable on every disk in $\C$, and hence it is globally stable. By Theorem \ref{thm: main_C_case}, $h_f$ is holomorphic, and by (2) in Lemma \ref{lem: properties preserved}, so is $h$. 
\end{proof}
Minor modifications of the proof above can be used to strengthen Theorem \ref{thm: main}, but the results we've found are less clean to state. We include one such result here.
\begin{thm}\label{thm: main_extended}
    For a discrete subset $P\subset \C$, let $h:\C\backslash P\to \R^n$ be a branched minimal immersion with Weierstrass-Enneper data $(\alpha_1,\dots, \alpha_n)$ such that each $\alpha_i$ is of the form
$$\alpha_i = f_i + g_{i,j},$$ where $f_i$ is meromorphic and $g_{i,j}(z)$ is an infinite Laurent expansion $$g_{i,j}(z)=\sum_{k=1}^\infty b_{i,j,k} (z-p_j)^{-k}$$ with only finitely many zeros. Then $h$ is stable if and only if it is holomorphic up to a rigid motion.
\end{thm}
Unfortunately, in view of the Picard theorems, this constraint on the $g_{i,j}$'s is quite strong.
\begin{proof}
    Multiplying $(\alpha_1,\dots, \alpha_n)$ by $\prod_{i,j} g_{i,j}^{-1}$, we obtain a tuple of meromorphic $1$-forms on $\C$. Note that each zero of each $g_{i,j}$ contributed a pole, and that's why we need the finite zeros assumption. One can then multiply by a holomorphic function, say $F$, as in the proof above, to get a tuple of holomorphic $1$-forms on $\C$. From here, the proof is identical to the proof of Theorem \ref{thm: main}, but one uses $F\prod_{i,j} g_{i,j}^{-1}$ in place of what we called $f$.
\end{proof}

\subsection{Covering stability}
Finally, we use our previous results to prove Corollary \ref{thm: FS}.

\begin{proof}[Proof of Corollary \ref{thm: FS}]
Let $\Sigma$, $P$, and $F:\Sigma\backslash P\to \R^n$ be as in the statement of the corollary. Throughout, $\Sigma\backslash P$ is equipped with the complex structure induced by $F$. Due to completeness and finite total curvature, by \cite[Theorem 1]{CO}, the ends corresponding to $P$ are ordinary punctures, and hence the complex structure extends to $\Sigma$. Henceforth, we identify $\Sigma$ with $\C/\Gamma$, for some discrete additive subgroup $\Gamma\subset \C$ (vertical translations if $\Sigma=\C^*$, a lattice if it's a torus), so that the projection map $\pi:\C\to \C/\Gamma=X$ is the universal covering.
    
    As we've previously commented, holomorphic implies covering stable. Assume that $F$ is not holomorphic up to a rigid motion. The map $\pi$ restricts to a (infinite sheeted) covering map $\pi:\C\backslash \hat{P}\to \Sigma$, where $\hat{P}=\pi^{-1}(P_0).$ Since $F$ is locally complete with finite total curvature around $P$, the map $$F\circ \pi: \C\backslash \hat{P}\to \R^n$$ satisfies the hypothesis of Theorem \ref{thm: main}. Thus, by Theorem \ref{thm: main}, $F\circ \pi$ is unstable. That is, there is a compact subset $\Omega\subset \C\backslash \hat{P}$ and a variation $W$ compactly supported on $\Omega$ such that $W$ lowers the area of $F\circ \pi$ to second order. This is not enough to conclude the instability of the original map $F$, since if $\gamma(\Omega)\cap\Omega$ is non-empty for some $\gamma\in \Gamma$, then $W$ will not descend to a variation on $\Sigma\backslash P$. However, we're now prepared to see that $F$ is not covering stable. For every positive integer $n$, consider the action of the subgroup $\Gamma_n:=n\Gamma$ on $\C$, through which the map $\pi: \C\backslash \hat{P}\to \Sigma\backslash P$ descends to a finite covering map $$p_n:\Sigma_n:=(\C\backslash \hat{P})/\Gamma_n\to (\C\backslash \hat{P})/\Gamma=\Sigma\backslash P$$ ($n$-sheeted if $\Sigma=\C^*$, $n^2$-sheeted in the torus case). For $n$ sufficiently large, say, strictly larger than the diameter of $\Omega$ over the minimal norm of a generator of $\Gamma$, $\gamma(\Omega)\cap\Omega=\emptyset$ for all $\gamma\in \Gamma_n$. Hence, $\Omega$ projects to a compact subset of $\Sigma_n$, and $W$ descends to a destabilizing variation on this subset. We deduce that $\Sigma_n$ is unstable, and hence $\Sigma\backslash P$ is not covering stable.
\end{proof}

\section{Destabilization radii}
To prove Theorem \ref{thm: radius} about minimal surfaces in $\R^3$, we give a simplified version of the proof of Theorem \ref{thm: main_C_case} and extract a radius on which we have instability. 

Recall that for a branched minimal immersion to $\R^3$, the Weierstrass-Enneper data $(\alpha_1,\alpha_2,\alpha_3)$ can be repackaged as the data of a holomorphic $1$-form $f$ and a meromorphic function $g$ such that $fg^2$ is holomorphic. The relation is  
$$(\alpha_1,\alpha_2,\alpha_3)=\left(\frac{1}{2}f(1-g^2),\frac{i}{2}f(1+g^2),fg\right).$$ To distinguish, we refer to the pair $(f,g)$ as the Weierstrass-Enneper representation instead of the Weierstrass-Enneper data. When the source is $\C$, we view $f$ as a holomorphic function. Theorem \ref{thm: radius} provides an estimate on the destabilization radius in terms of $f$ and $g$.

Before proving Theorem \ref{thm: radius}, we just need one lemma. Let $\alpha=\sum_{j=0}^\infty c_jz^j$ be a holomorphic function. Define, for $m>0$ and $\gamma\in\mathbb{C}^*$,
$$C(\alpha,\gamma,m):=\sum_{j=0}^{m-1}(m-j)\left(\Re(\gamma^2c_jc_{2m-j})+\vert\gamma c_j\vert^2\right).$$
\begin{lemma}
\label{lemma: destabilization_series}
Let $h=(h_1,\dots, h_n):\overline{\mathbb{D}}\to \R^n$ be a branched minimal immersion such that each $h_i$ has no zeros on $\partial\mathbb{D}.$ For each $i$, let $\alpha_i = \frac{\partial h}{\partial z}.$ Suppose that there exists an integer $m>0$ and $\gamma\in\mathbb{C}^*$ such that
    $$\sum^n_{i=1}C(h_i,\gamma,m)<0.$$
    Then, feeding $\varphi(z)=\gamma z^{-m}$ into Theorem \ref{thm: thmc} destabilizes $h$.
\end{lemma}
\begin{proof}
    One can read off from Lemma \ref{lem: main computation} that $\mathcal{F}_h(\varphi)$ reduces to $\sum^n_{i=1}C(h_i,\gamma,m)$, and then apply Theorem \ref{thm: thmc}; we omit the details. One can also just follow the proof of Theorem D from \cite{MS}, which is Lemma \ref{lemma: destabilization_series} in the case where every $h_i$ is a polynomial.
\end{proof}

\begin{proof}[Proof of Theorem \ref{thm: radius}]
    Let $h:\overline{\mathbb{D}}_r\to \R^3$ be a branched minimal immersion as in the statement of the theorem, with Weierstrass-Enneper representation $(f,g).$ Write $f=\sum_{j=k}^{\infty}b_jz^j$ and $g=a_0+\sum_{j=m}^{\infty}a_jz^j$, with $m>0,$ $a_m,b_k\neq 0.$ To apply Lemma \ref{lemma: destabilization_series}, we reparametrize $h|_{\overline{\mathbb{D}}_r}$ to $\overline{\mathbb{D}}$ via $h^r:=h(r\cdot).$ Write $h^r=(h_1^r,\dots, h_n^r).$ 

 We apply Lemma \ref{lemma: destabilization_series} with the integer $k+m>0$ and  $\gamma\in\mathbb{C}^*$ to be chosen later. Denote by $c^i_j$ the $j^{th}$ Taylor coefficient of the functions $h_{i,z}$. We view this expression from Lemma \ref{lemma: destabilization_series} as a polynomial in $r$:
\begin{align*}
    \sum_{i=1}^{3}&C(h_i^r,\gamma,k+m)=\sum_{i=1}^{3}\sum_{j=k}^{k+m-1}(k+m-j)\left(\operatorname{Re}(\gamma^2c^i_jc^i_{2k+2m-j})r^{2k+2m}+\vert\gamma c^i_j\vert^2r^{2j}\right)\\
    &= 
m \underbrace{\sum_{i=1}^3 \operatorname{Re}\left(\gamma^2 c^i_k c^i_{k+2m}\right)}_{=:A_{k,m}} r^{2k+2m}
+ \sum_{j=1}^{m-1} (m-j) \underbrace{\sum_{i=1}^3 \operatorname{Re}\left(\gamma^2 c^i_{k+j} c^i_{k+2m-j}\right)}_{=:B_{k,m}^j}r^{2k+2m} \\
&+ \sum_{j=0}^{m-1} (m-j) \underbrace{\sum_{i=1}^{3} \lvert \gamma c^i_{k+j} \rvert^2}_{=:C_{k}^j}r^{2k+2j} .
\end{align*}
We now work toward finding expressions for $A_{k,m}$, $B_{k,m}^j$, $1\leq j \leq m-1$, and $C_k^j$, $0\leq j \leq m-1$, in terms of the Taylor coefficients of $f$ and $g$. We first write out
$$h_1^r(z)=\frac{1}{2}\sum_{j=k}^{\infty}(b_j-D_j)r^{j}z^j, \qquad h_2^r(z)=\frac{i}{2}\sum_{j=k}^{\infty}(b_j+D_j)r^{j}z^j, \qquad h^r_3(z)=\sum^\infty_{j=k}F_jr^{j}z^{j},$$
where $D_j:=\sum_{i_1+i_2+i_3=j} b_{i_1}a_{i_2}a_{i_3}$ and   $F_j:=\sum_{i_1+i_2=j}a_{i_1}b_{i_2}$. Then, a direct computation shows that for any $0\leq j\leq m-1$,
\begin{align*}
    \sum_{i=1}^3 c^i_{k+j}c^i_{k+2m-j}&=\frac{1}{2}(b_{k+j}-D_{k+j})(b_{k+2m-j}-D_{k+2m-j})\\
    &-\frac{1}{2}(b_{k+j}+D_{k+j})(b_{k+2m-j}+D_{k+2m-j})+F_{k+j}F_{k+2m-j}\\
    &=-\frac{1}{2}(b_kD_{k+2m}+b_{k+2m}D_k)+F_kF_{k+2m}.
\end{align*}
Using that $a_j=0$ for $1\leq j\leq m-1$ and $b_j=0$ for $j\leq k-1$, we have the following equalities: for $1\leq j\leq m-1$,
$$D_{k+2m-j}=2a_0F_{k+2m-j}-b_{k+2m-j}a_0,$$
and for $j=0$
$$D_{k+2m}=b_ka_m^2+2a_0F_{k+2m}-b_{k+2m}a_0.$$
As well, we have $D_{k+j}=b_{k+j}a_0^2$ and $F_{k+j}=b_{k+j}a_0$ for any $0\leq j\leq m-1$. 

Inserting the expressions above into the formula for $\sum_{i=1}^3 c^i_{k+j}c^i_{k+2m-j}$, we obtain expressions that are used to understand $A_{k,m}$, $B_{k,m}^j$, and $C_k^j$. For $j=0$, the terms in $F_{k+2m}$ cancel out and we obtain $-\frac{1}{2}b_k^2a_0^2$. Therefore, $A_{k,m}=-\frac{1}{2}\Re(\gamma^2 b_k^2a_0^2)$.
For $j\geq 1$, all terms of the sum cancel out and, from we conclude that $B^j_{k,m}=0$ for all $j$. Finally, for $C_{k}^j$, for all $j$, we have
\begin{align*}
    C^j_{k}&=\frac{1}{4}\lvert b_{k+j}-D_{k+j}\rvert^2 + \frac{1}{4}\lvert b_{k+j}+D_{k+j}\rvert^2 + \lvert F_{k+j}\rvert^2\\
    &=\frac{1}{4}\lvert\gamma b_j\rvert^2(\vert1-a_0^2\vert^2+\vert1+a_0^2\vert^2+4\vert a_0\vert^2)
    =\frac{1}{2}\lvert\gamma b_j\rvert^2(1+\lvert a_0\rvert^2)^2.
\end{align*}
Returning to the expression from Lemma \ref{lemma: destabilization_series}, at this stage we can write $$\sum_{i=1}^3 C(h_i^r,\gamma,k+m)= -\frac{1}{2}\Re(\gamma^2 b_k^2a_0^2)r^{2m}+\frac{1}{2}\sum_{j=0}^{m-1}(m-j)\vert b_{k+j}\vert^2(1+\vert a_0\vert^2)^2r^{2j}.$$ 
Now, we are yet to choose $\gamma$. We choose it so that $\Re(\gamma^2 b_k^2a_0^2)=\vert \gamma b_k a_0\vert^2$. Then, the equation above becomes 
$$\sum_{i=1}^3 C(h_i^r,\gamma,k+m)= -\frac{1}{2}\vert \gamma b_k a_0\vert^2r^{2m}+\frac{1}{2}\sum_{j=0}^{m-1}(m-j)\vert b_{k+j}\vert^2(1+\vert a_0\vert^2)^2r^{2j}.$$ This is $\frac{1}{2}P_m(r)$, and the theorem follows from Lemma \ref{lemma: destabilization_series}, provided no $h_i$ has a zero on $\partial\mathbb{D}_r$. If some $h_i$ has a zero on $\partial\mathbb{D}_r$, then we can shrink $r$ ever so slightly, so that it is still larger than the smallest positive root of $P_m$, and then apply Lemma \ref{lemma: destabilization_series}.
\end{proof}

\bibliographystyle{plain}
\bibliography{bibliography}

\end{document}